\documentclass{amsart}

\usepackage[left=1.2in, right=1.2in, bottom=1.2in, top=1.2in]{geometry}
\usepackage{chngcntr}
\usepackage{amsfonts}
\usepackage{sseq}
\usepackage{caption}
\usepackage{amssymb}
\usepackage{mathrsfs}
\usepackage{xcolor,graphicx}
\usepackage{float}
\usepackage{amsmath}
\usepackage{url}
\usepackage{verbatim}

\newcommand{\locs}{\mathrm{Loc}}
\usepackage{xcolor} 
\usepackage{hyperref}
\definecolor{dark-red}{rgb}{0.5,0.15,0.15}
\definecolor{dark-blue}{rgb}{0.15,0.15,0.6}
\definecolor{dark-green}{rgb}{0.15,0.6,0.15}
\hypersetup{
    colorlinks=true, linkcolor=dark-red,
    citecolor=dark-blue, urlcolor=dark-green
}

\usepackage{xy}
\input xy
\xyoption{all}

\newcommand{\md}{\mathrm{Mod}}
\newcommand{\fun}{\mathrm{Fun}}

\newcommand{\loc}{\mathrm{Loc}}

\newcommand{\pic}{\mathrm{Pic}}

\newcommand{\gl}{\mathfrak{gl}}

\newcommand{\pics}{\mathfrak{pic}}

\usepackage{amsthm}
\usepackage{cleveref}
\crefformat{equation}{(#2#1#3)}

\renewcommand{\sp}{\mathrm{Sp}}
\newcommand{\st}{\mathrm{Stack}}
\newcommand{\otop}{\mathcal{O}^{\mathrm{top}}}

\renewcommand{\hom}{\mathrm{Hom}}

\numberwithin{equation}{section}
\newtheorem{lm}{lm}[section]

\newtheorem{corollary}[lm]{Corollary}

\newtheorem{theorem}[lm]{Theorem}
\newtheorem{proposition}[lm]{Proposition}

\renewcommand{\rightrightarrows}{\begin{smallmatrix} \to \\
\to \end{smallmatrix} }
\newcommand{\triplearrows}{\begin{smallmatrix} \to \\ \to \\ 
\to \end{smallmatrix} }

\theoremstyle{definition}
\newtheorem{definition}[lm]{Definition}

\renewcommand{\ell}{\mathrm{Ell}}

\theoremstyle{definition}

\newtheorem{construction}[lm]{Construction}

\newtheorem{remark}[lm]{Remark}

\newtheorem{example}[lm]{Example}

 \newcommand{\e}[1]{\mathbf{E}_{#1}}

\begin{document}

\title{Torus actions on  stable module categories, Picard groups, and localizing  
subcategories}
\date{\today}
\author{Akhil Mathew}
\address{Department of Mathematics, Harvard University, Cambridge, MA}
\email{amathew@math.harvard.edu} 

\thanks{The author is supported by the NSF Graduate Research
Fellowship under grant DGE-1144152 and thanks the Hausdorff Institute of
Mathematics for hospitality during the period during which most of this
writing was done.}

\begin{abstract}
Given an abelian $p$-group $G$ of rank $n$, we construct an action of the torus
$\mathbb{T}^n$ on the stable module $\infty$-category of $G$-representations
over a field of characteristic $p$. The homotopy fixed points are given by the
$\infty$-category of module spectra over the Tate construction of the torus.
The  relationship thus obtained arises from a Galois extension in the sense of
Rognes, with Galois group given by the torus. 	
As one application, 
we give a homotopy-theoretic proof of Dade's classification of endotrivial modules for
abelian $p$-groups. 
As another application, we give a slight variant of a key step in the 
Benson-Iyengar-Krause proof of the classification of 
localizing subcategories of the stable module category. 
\end{abstract}

\maketitle

\tableofcontents

\renewcommand{\st}[1]{\mathrm{StMod}_{#1}}
\section{Introduction}

\subsection{Stable module categories}

Let $G$ be a finite group and let $k$ be a field of characteristic $p$. 
Our object of study is the \emph{stable module category} of $G$, denoted
$\st{G}$. 
The objects of $\st{G}$ are the $k[G]$-modules, i.e.,
$G$-representations over $k$. 
Given two objects $M, N \in \st{G}$, we have
\[ \hom_{\st{G}}(M, N) \stackrel{\mathrm{def}}{=} \hom_{k[G]}(M, N)/ \sim,  \]
where we take the quotient by the subspace of all $k[G]$-module maps $M \to N$ which factor through  a
projective $k[G]$-module. 
The stable module category $\st{G}$ has a symmetric monoidal structure given
by the $k$-linear
tensor product of representations. 

\begin{definition} 
\label{endotrivdef}
An object $M \in \st{G}$ with $\dim_k M < \infty$ is called \textbf{endotrivial} if its  endomorphism
ring $M \otimes_k M^{\vee}$ is equivalent in  $\st{G}$ to the unit $k$;
equivalently, if $M$ belongs to the Picard group of $\st{G}$. 
\end{definition} 

The unit $k$ is obviously endotrivial. 
Additional examples 
of endotrivial modules arise from the fact that $\st{G}$ is actually a
\emph{triangulated} category, where the suspension of an object $M \in \st{G}$
is obtained by choosing an embedding $M \hookrightarrow M'$ for $M'$ free and
then taking the quotient $M'/M$. 
The triangulated and tensor structures are appropriately compatible, so the
suspensions $\Sigma^i k$ for $i \in \mathbb{Z}$ are also endotrivial. 

\begin{theorem}[{Dade \cite[Th. 10.1]{DadeII}}] 
Suppose $G$ is an abelian $p$-group. Then every endotrivial module is equivalent in $\st{G}$
to $\Sigma^i k$ for some $i \in \mathbb{Z}$. 
\end{theorem} 

The Picard group of $\st{G}$ for a general finite $p$-group $G$ is an important
object of study in modular representation theory, and has been completely
calculated in the work of Carlson--Thevenaz \cite{CT2, CT, CT3}. 
One of our goals in this paper is to  give  a homotopy-theoretic proof of
Dade's theorem, using descent theory. 
We note that related methods of tensor-triangulated categories have been used to
study Picard groups of stable module categories 
in \cite{PicTT} and in particular to obtain information about Picard groups after inverting the
characteristic.

\subsection{Structured ring spectra}

\label{subsec:structured}
We now give an overview of our proof of Dade's theorem. 
Suppose we were working not in the stable module category but the ordinary
representation category. Then we could make the following argument to show
that the Picard group is trivial: every
element of the abelian $p$-group $G$ induces multiplication by a scalar in $k^{\times}$ on our
representation. Since $k^{\times}$ is $p$-torsion free, every element of $G$
acts trivially and we are done. 

One can start this argument in the stable module category. However, it is 
now only true that every element of the group acts as the identity \emph{up to
homotopy} (i.e., modulo projectives) on an invertible object. This is insufficient to conclude that the
object is trivial. Nonetheless, we will still be able to carry out a more
sophisticated version of this argument, for which the language of
$\infty$-categories and higher algebra \cite{HTT, HA} will be essential. 
In particular, we will henceforth work with the $\infty$-categorical
enhancement of $\st{G}$ (which by abuse of notation we will denote by the same
symbol).

Our starting point is that $\st{G}$ is equivalent to the $\infty$-category of
module spectra over the \emph{Tate construction} $k^{tG}$, an
$\e{\infty}$-ring spectrum obtained as a suitable localization of the cochain
algebra $k^{hG} \simeq  F(BG_+, k)$.
In particular, the group of stable equivalence classes of endotrivial modules is
isomorphic to  
the Picard group of the $\e{\infty}$-ring spectrum $k^{tG}$. 
Dade's theorem states that the Picard group of $k^{tG}$ is generated by the
suspensions of the unit if $G$ is abelian.

Picard groups of structured ring spectra form a recurring topic in stable
homotopy theory, 
starting with the observation of Hopkins that the Picard groups of the $L_n$
and $K(n)$-localized stable homotopy categories contain significantly more than
suspensions (or algebraically flat objects). 
Given an $\e{\infty}$-ring spectrum $R$, a general theme is that the Picard
group will be easier to understand if the homotopy groups of $R$ are
homologically simple. 
For example, if $\pi_*(R)$ is \emph{regular} and concentrated in even degrees,
the Picard group of $R$ can be determined purely algebraically, cf. 
\cite{BakerRichter}, \cite[Th. 6.4]{HillMeier}. 
A general approach \cite{GL} to the computation of Picard groups of $\e{\infty}$-ring spectra
is to decompose the $\infty$-category of $R$-modules into $\infty$-categories of modules over
$\e{\infty}$-rings with better-behaved homotopy groups, and apply a descent
spectral sequence. We refer to \cite{MS, HMSPic} for examples of this approach.

\subsection{Our approach}

If $G$ is abelian, then the $\e{\infty}$-ring $k^{tG}$ has fairly
complicated homotopy groups, homologically. For instance, if $p$ is odd and $G$
has rank one, then 
\[ \pi_* k^{tG} \simeq E(\alpha_{-1}) \otimes_k k[\beta_{-2}^{\pm 1}],  \]
i.e., the tensor product of an exterior algebra and a Laurent polynomial
algebra, 
and this has infinite homological dimension.
When $G$ has higher rank, the positive homotopy groups of $\pi_*( k^{tG})$
are entirely square-zero. 

A direct approach using descent theory to the Picard group of $k^{tG}$
is problematic because of the presence of these exterior classes;
instead, we will use a sort of reverse descent, following ideas that we used
in \cite[\S 9]{galois} to study the Galois group. 
The main step is the following result. 

\begin{theorem}
If $G $ is an abelian $p$-group of rank $n$, then one can construct an action of the $n$-torus
$\mathbb{T}^n$ on the symmetric monoidal $\infty$-category $\st{G}$ 
such that 
the homotopy fixed points $\st{G}^{h\mathbb{T}^n}$ 
are given by the 
analog of the stable module $\infty$-category for the torus.
\end{theorem}

The above arises from a faithful $\mathbb{T}^n$-Galois extension in the sense
of Rognes \cite{rognes} that runs $k^{t \mathbb{T}^n} \to k^{tG}$. 
The Picard group of the $\e{\infty}$-ring $k^{t \mathbb{T}^n}$ can be calculated directly using descent theory. 
In particular, we can prove the analog of Dade's theorem for 
$\st{G}^{h\mathbb{T}^n}$ relatively easily.
We will show using an obstruction-theoretic calculation  (\Cref{descpictool}) that any
invertible object in
$\st{G}$ can be ``descended'' to the homotopy fixed points $\st{G}^{h
\mathbb{T}^n} \simeq \md( k^{t\mathbb{T}^n})$. Since the Picard group 
of the latter is cyclic, this will complete the proof. This relies on
techniques with descent spectral sequences following \cite{MS}. 

As a result, we will also obtain a partial local version of Dade's theorem. 
Choose an identification $\mathrm{Proj} (H^{\mathrm{even}}(G;
k))_{\mathrm{red}} \simeq \mathbb{P}^{n-1}_k$.
Let
$U \subset \mathbb{P}^{n-1}_k$ be a Zariski open subset. 
We say that a morphism $M \to N$ in $\st{G}$ is an \emph{$U$-equivalence}
if its cofiber has cohomology supported on the complement of $U$.
We can then form an associated Bousfield localization
$(\st{G})_{U}$ where we invert all $U$-equivalences. 
Then we will prove:

\begin{theorem} 
Suppose $U \subset \mathbb{P}^{n-1}_k$ is affine and $p > 2$. 
Let $M \in \st{G}$ be a compact object (i.e., one represented by a
finite-dimensional $k[G]$-representation). Suppose that the $U$-localization
of $M$ is invertible in $(\st{G})_U$. Then the $U$-localization of $M$
is equivalent to the $U$-localization of $\Sigma^i k$ for some $i$.  
\end{theorem} 

\subsection{Thick and localizing subcategories}

The idea of comparing the $\e{\infty}$-ring $k^{t G}$ (or $k^{hG}$) with an $\e{\infty}$-ring
whose homotopy groups do not have the exterior algebra classes is not new: it
is used prominently in the 
 stratification
\cite{BIK} of localizing subcategories of the stable module category of an
$p$-group. The main result runs as
follows: 

\begin{theorem}[Benson-Iyengar-Krause \cite{BIK}] 
\label{BIKthm}
Let $G$ be a finite $p$-group. 
The localizing subcategories of $\md( k^{hG})$ are in bijection with the
subsets of the set of homogeneous prime ideals in $\pi_*
(k^{hG})$. The localizing subcategories of the stable module category
$\st{G} \simeq \md( k^{tG})$ are in
bijection with subsets that do not contain the irrelevant ideal.
\end{theorem} 

Earlier work of Benson-Carlson-Rickard \cite{BCR} classifies the
\emph{thick} subcategories of the compact objects of $\st{G}$.
The above result has been extended in the work of Stevenson \cite{Stevenson}, for
example for the singularity categories of complete intersection local rings. 

Suppose $G$ is in fact {elementary abelian}, $G \simeq C_p^n$.
Then
Benson-Iyengar-Krause prove \Cref{BIKthm} by replacing $k^{h C_p^n}$
with a different $\e{\infty}$-ring $R$,
given by   
the (derived) ring of functions
on the classifying stack of the $k$-group
\emph{scheme} $ \alpha_p^n$. The underlying $\e{1}$-rings of $R$ and
$k^{hC_p^n}$  are equivalent, so it
suffices to classify localizing subcategories of 
$\md(R)$. 
Now the $\e{\infty}$-ring $R$ receives a map $R' \to R$, where $\pi_*(R')$ 
has only polynomial classes (cf. \cite[\S 7]{BIK}). It is shown 
(cf. \cite[Th. 4.4]{BIK}) 
how to use a classification result for
localizing subcategories of $\md(R')$ to obtain one for $\md(R)$. The
classification of localizing subcategories of $\md(R')$ can be proved  using a technique with
``residue fields'' which goes back to \cite{HPSt}, as $\pi_*(R')$ is
\emph{regular}. 

In this paper, we study a $\mathbb{T}^n$-Galois extension of $\e{\infty}$-rings $k^{h \mathbb{T}^n}
\to k^{h C_p^n}$. This gives
an $\e{\infty}$-approximation to $k^{h C_p^n}$ that only sees the
polynomial classes.  On $\e{1}$-rings, this extension is equivalent to the
one of \cite{BIK}, but it is constructed using topology rather than
graded Hopf algebras. 

We will prove a general result that for a faithful $\mathbb{T}^n$-Galois extension of
$\e{\infty}$-rings $R_1 \to R_2$, the localizing subcategories of $\md(R_1)$ and
$\md(R_2)$ are in canonical bijection. 
As a result, we can reduce the classification of localizing
subcategories of $\md(k^{h C_p^n})$ to the analog in $\md( k^{h
\mathbb{T}^n})$, where one can use the method of residue fields. 
This gives a slightly different approach to some of the technical steps in
\cite{BIK}, which does not require modifying the $\e{\infty}$-structure. 

\subsection{Acknowledgments}
I would like to thank Ben Antieau, John Greenlees, Mike Hill, Mike Hopkins, Srikanth Iyengar, Jacob Lurie, Julia Pevtsova, and
especially Vesna Stojanoska  for helpful discussions. 

\section{Stable module categories and Tate spectra}

\subsection{Construction of the stable module $\infty$-category}

We start by reviewing the construction of the stable module $\infty$-category
of a finite group. 
Although these results are not new, we have spelled out some of the details
for the convenience of the reader. 

\begin{construction}
Let $G$ be a finite group 
and let $k$ be a field of characteristic $p>0$. Consider the group ring $k[G]$
and the symmetric monoidal category 
$\mathcal{C}$ of (discrete) $k[G]$-modules, equipped with the $k$-linear
tensor product. 

The category $\mathcal{C}$ admits a combinatorial model structure \cite[\S 2.2]{hovey}
where: 
\begin{enumerate}
\item The fibrations are the surjections.  
\item The cofibrations are the injections. 
\item The weak equivalences are the {stable equivalences.} Given a
morphism $f\colon V_1 \to V_2$ in $\mathcal{C}$, one says that it is a \emph{stable
equivalence} if 
there exists a morphism $g\colon V_2 \to V_1$ such that the endomorphisms $g
\circ f - \mathrm{id}_{V_1}$
and $f \circ g - \mathrm{id}_{V_2}$, of $V_1, V_2$ respectively, each factor through
a projective $k[G]$-module.
\end{enumerate}
One checks in addition that $\mathcal{C}$ is a symmetric monoidal model category.
\end{construction}

Note that every object in the model category $\mathcal{C}$ is cofibrant-fibrant. It is shown 
in \cite[\S 2.2]{hovey} that the morphisms in \emph{homotopy category}
$\mathrm{Ho}(\mathcal{C})$ are given by the classical stable module category,
i.e., 
\begin{equation} \label{classst}\hom_{\mathrm{Ho}(\mathcal{C})}(V_1, V_2) \simeq \hom_{\mathcal{C}}(V_1,
V_2)/\sim,\end{equation}
where we identify two morphisms if their difference factors through a
projective. 

\begin{definition} 
The stable module $\infty$-category $\st{G}$ is the
$\infty$-categorical localization $\mathcal{C}[\mathcal{W}^{-1}]$ for
$\mathcal{W} \subset \mathcal{C}$ the weak equivalences \cite[Def.
1.3.4.15]{HA}. This inherits the
structure of a symmetric monoidal $\infty$-category \cite[Prop. 4.1.3.4]{HA}. 
\end{definition} 

\newcommand{\perf}{\mathrm{Perf}}
The stable module $\infty$-category also has another description as a
Bousfield localization which
we review. We use  \cite{MNN15} as a reference for some of these ideas in the
context of equivariant stable homotopy theory, but they  are of course much
older (e.g., \cite{HPSt}). 
\begin{construction}
Consider the symmetric monoidal $\infty$-category $\fun(BG, \perf(k))$ of perfect
$k$-module spectra equipped with a $G$-action. We form the
$\mathrm{Ind}$-completion $\mathrm{Ind}(\fun(BG, \perf(k)))$ and, inside here,
the $A^{-1}$-localization (cf. \cite[\S 3]{MNN15}) for $A = F(G_+, k) \in
\mathrm{CAlg}(\fun(BG, \perf(k)))
$, which we denote $L_{A^{-1}}
\mathrm{Ind}(\fun(BG, \perf(k)))$. 
\end{construction}

These constructions have been extensively studied in the literature for any
noetherian ring, not only group rings, starting with the
work of Krause \cite{Krausestable}. 
We refer to the work of Benson-Krause \cite{BK} for a construction of
$\mathrm{Ind}( \fun(BG, \perf(k)))$
via a model category of \emph{complexes of injectives}. 
Finally, these $\infty$-categories have been studied by Gaitsgory in the more general
setting of DG-algebras and DG-schemes \cite{indcoh}.

\begin{theorem} 
There is an equivalence of symmetric monoidal $\infty$-categories
between $\st{G}$ and $L_{A^{-1}} \mathrm{Ind}(\fun(BG, \perf(k)))$. 
\end{theorem} 
\begin{proof} 
The main point is the construction of the functor in this language. 
Let $\mathcal{C}$ be, as before, the model category of (discrete)
$k[G]$-modules. 
Note that $\mathcal{C} = \mathrm{Ind}(\mathcal{C}')$ for $\mathcal{C}'$ the
category of finitely generated (discrete) $k[G]$-modules. 

Here $\perf(k$) denotes the $\infty$-category of perfect $k$-module
spectra. 
Let $\perf^{\heartsuit}(k) \subset \perf(k)$ be the full subcategory spanned by the
discrete, finite-dimensional $k$-vector spaces.
We have a
functor
\[  \mathcal{C}' \to \fun(BG, \perf^{\heartsuit}(k)) \subset \fun(BG, \perf(k)) \]
which extends by $\mathrm{Ind}$-completion to a symmetric monoidal functor
\[ \Phi \colon \mathcal{C}  \to \mathrm{Ind} (\fun(BG, \perf(k))).\]
Here $\Phi$ commutes with filtered colimits. 
The composite
\[  \mathcal{C}  \xrightarrow{\Phi} \mathrm{Ind} (\fun(BG, \perf(k))) \to 
L_{A^{-1}}\mathrm{Ind} (\fun(BG, \perf(k)))
\]
sends projectives to zero, so it respects weak equivalences and factors through
a symmetric monoidal functor
\[ \varphi \colon \st{G} \simeq \mathcal{C}[\mathcal{W}^{-1}] \to
L_{A^{-1}}\mathrm{Ind} (\fun(BG, \perf(k))). \]
This is the desired functor. 

We claim that the above functor $\varphi $ commutes with homotopy colimits. 
We first show that $\varphi$ is exact. 
Suppose given a cofiber sequence $M' \to M \to M''$ in $\st{G}$.  
Using the model structure on $\mathcal{C}$, 
we can represent this cofiber sequence by a short exact sequence in $\mathcal{C}$, e.g., 
\begin{equation} \label{ses} 0 \to M' \to
M \to M'' \to 0.\end{equation} 
The short exact sequence is a filtered colimit in $\mathcal{C}$ of short exact
sequences 
\begin{equation} \label{sessmall}
0 \to M'_\alpha \to M_\alpha \to M''_\alpha \to 0, \end{equation}
which belong to
$\mathcal{C}'$, i.e., $M_\alpha$ is finitely generated.  
Observe that $\Phi$ sends each of the short exact sequences 
\eqref{sessmall} to a cofiber sequence in $\fun(BG, \perf(k))$. Therefore, $\Phi$ (and therefore
$\varphi$) sends \eqref{ses} to a cofiber sequence too. 
Next, we observe that $\varphi$ respects arbitrary direct sums (since $\Phi$
does). It follows that $\varphi$ is cocontinuous \cite[Prop. 1.4.4.1]{HA}.

To check that $\varphi$ is an equivalence, we 
check that it is fully faithful on the homotopy category. 
Since $\fun(BG, \perf(k))$ is generated as a thick subcategory by $\fun(BG,
\perf^{\heartsuit}(k))$,
essential surjectivity will then be automatic. 

Note that $\st{G}$ is stable and the objects represented by finitely generated
$k[G]$-modules are easily seen to be compact (note that compactness can be
checked at the level of the homotopy category, cf. \cite[Prop.~1.4.4.1]{HA}).  
Using duality, one reduces to showing that if $M \in \mathcal{C}'$, then
the natural map
\begin{equation}
\label{compmap}
\hom_{\mathrm{Ho}(\mathcal{C})}(\mathbf{1}, M) \xrightarrow{\varphi}
\pi_0 \hom_{
L_{A^{-1}} \mathrm{Ind} (\fun(BG,
\perf(k)))}( \mathbf{1}, M).
\end{equation}
is an isomorphism.
However, both are known to be the Tate cohomology $\hat{H}^0(G; M)$:
\begin{enumerate}
\item In $\mathrm{Ho}(\mathcal{C})$, this is now an easy calculation using the description of
the homotopy category \eqref{classst}. 
\item We need to show that in 
$L_{A^{-1}} \mathrm{Ind} (\fun(BG,
\perf(k)))$, maps from the unit to the $A^{-1}$-localization of $M$
are as claimed. This follows from the explicit description of
$A^{-1}$-localization, cf. \cite[\S 3]{MNN15} and the following subsection. 

We can also spell out the details directly. 
The $A^{-1}$-localization of $M$
in $\mathrm{Ind} (\fun(BG,
\perf(k)))$
can be computed as 
the cofiber of a map \begin{equation}\label{A1local} | M \otimes (\mathbb{D} A)^{\otimes \bullet + 1}| \to
M ,
\end{equation} (cf. \cite[\S 3]{MNN15}) for $\mathbb{D}$ denoting duality. Since $\hom_{\mathrm{Ind} (\fun(BG,
\perf(k)))}( \mathbf{1}, \cdot)$ commutes with arbitrary colimits, one now
calculates directly that 
the mapping spectrum
$$\hom_{\mathrm{Ind} (\fun(BG,
\perf(k)))}( \mathbf{1},
| M \otimes (\mathbb{D} A)^{\otimes \bullet + 1}|
)$$
is connective with $\pi_0$ given by the coinvariants $M_G$, and that the map 
\eqref{A1local} induces the norm map on $\pi_0$. 
This implies the desired description of the right-hand-side of \eqref{compmap}.
\end{enumerate}
\end{proof} 
\newcommand{\gsp}{\mathrm{Sp}_G}
\subsection{Connection with Greenlees-May}
In the stable module $\infty$-category $\st{G}$, 
endomorphisms of the unit are given by the \emph{Tate construction.} We review the connection
between $\st{G}$ and the Greenlees-May definition of the Tate
construction \cite{GMTate}. 

Fix a finite group $G$.

\begin{definition}[{\cite{GMTate}}] 
Let $\gsp$ denote the symmetric monoidal $\infty$-category of $G$-spectra. 
Given a $G$-spectrum $X$, the \emph{Tate construction} is
the $G$-spectrum $\widetilde{EG} \wedge F(EG_+, X)$. We will write $X^{tG}$ for the $G$-fixed
points $(\widetilde{EG} \wedge F(EG_+, X))^G$. When the context is
clear, we  will also call this the Tate construction. 
\end{definition} 

Let $A = F(G_+, S^0_G)$; this is a commutative algebra object in $\gsp$. 
As in \cite[\S 6]{MNN15}, we recall that $EG_+$ is the $A$-acyclization (cf.
\cite[\S 3]{MNN15}) of the unit and $\widetilde{EG}$ is the
$A^{-1}$-localization. In particular, $F(EG_+, X)$ is the $A$-completion of 
$X$.

\begin{construction}
Let $\underline{k} \in \mathrm{CAlg}(\gsp)$ denote the Borel-equivariant (or cofree) form of $k$, i.e.,
the $G$-spectrum representing Borel-equivariant cohomology with coefficients
in $k$. 
The symmetric monoidal $\infty$-category
$\perf_{\gsp}(\underline{k})$
of compact $\underline{k}$-modules in $\gsp$ is identified with $\fun(BG,
\perf(k))$. This follows because $\fun(BG,
\perf(k))$ is generated as a thick subcategory by the permutation modules $\{k
\wedge (G/H)_+\}_{H \leq G}$ and
in view of \cite[\S 6.3]{MNN15}. 
In particular, it follows by applying $\mathrm{Ind}$-completion that: 

\begin{theorem} 
We have an equivalence of symmetric monoidal $\infty$-categories:
\begin{equation} \md_{\gsp}(\underline{k}) \simeq\mathrm{Ind}(\fun(BG,
\perf(k))) .
\end{equation}
\end{theorem} 
\end{construction}

Under this, the two objects that we have denoted $A$ correspond. 
Applying $A^{-1}$-localization, one finds:
\begin{corollary} 
We have an equivalence of symmetric monoidal $\infty$-categories $\st{G} \simeq
\md_{\gsp}(\underline{k} \wedge \widetilde{EG})$.
The endomorphisms of the unit are given by the $\e{\infty}$-algebra $k^{tG} =
(\underline{k} \wedge \widetilde{EG} )^G$.
\end{corollary}

We specialize to the case where $G$ is a $p$-group. We find: 

\begin{theorem}[cf. Keller \cite{keller}] 
\label{stmodtate}
If $G$ is a $p$-group, then there is an equivalence 
of symmetric monoidal $\infty$-categories
between $\st{G}$ and $\md(k^{tG})$ for $k^{tG}$ the Tate construction of $G$. 
\end{theorem} 
\begin{proof} 
In this case, $\st{G}$ is
generated as a localizing subcategory by the unit because 
$\fun(BG,
\perf(k))$ is generated as a thick subcategory by the unit. This follows from
the fact that any nontrivial $G$-representation over $k$ has a nontrivial
fixed vector. 

By the symmetric monoidal version of the Schwede-Shipley
theorem
\cite[Prop. 7.1.2.7]{HA}, one obtains  a symmetric monoidal equivalence 
\begin{equation} \mathrm{Ind}(\fun(BG,
\perf(k))) \simeq \md_{\gsp}(\underline{k}) \simeq \md( k^{hG}) , \end{equation} where $k^{hG} \simeq F(BG_+, k)$ is the
$\e{\infty}$-ring of $k$-valued cochains on the classiyfing space $BG$. 
The result follows. 
\end{proof} 

\begin{example} 
\label{tateex}
Suppose $G$ is an abelian $p$-group of rank $n$. In this case, $\pi_* k^{hG} $
contains a   polynomial algebra  $k[x_1, \dots, x_n]$
on classes $x_i$ with $|x_i| = -2$. 

We will need to know that in this case, the Tate construction
$k^{tG}$ can be identified with the localization of $k^{hG}$ away from the
\emph{ideal} $(x_1, \dots, x_n)$ (cf. \cite{GM}). Equivalently, the localizing
subcategory of $\md(k^{hG}) \simeq \mathrm{Ind}( \fun(BG, \perf(k)))$ generated
by the iterated cofiber $k^{hG}/(x_1, \dots, x_n)$ (using the first description of this
$\infty$-category) is equal to that generated by $k \wedge
G_+$ (using the second). 

Although this description is classical, we will describe the argument in the
case where $G$ is \emph{elementary abelian}. 
The $x_i$'s are the
Euler classes of complex line bundles for a linearly independent set of maps
$G \twoheadrightarrow \mu_p \subset S^1$.  
In more detail, given a complex character $\chi$ of $G$, we form the unit
sphere $S(\chi)$, the Euler sequence $S(\chi)_+ \to S^0 \to S^{\chi}$, and the 
induced cofiber sequence in $\fun(BG, \perf(k))$
\[ S(\chi)_+ \wedge k \to k \to \Sigma^2 k.  \]
In $\md(k^{hG})$, we obtain a corresponding cofiber sequence
\[ (S (\chi)_+ \wedge k)^{hG} \to k^{hG}\xrightarrow{e(\chi)} \Sigma^2
k^{hG},  \]
for $e( \chi) \in \pi_{-2}(k^{hG})$ the Euler class of the character. 
From this, it follows that the cofiber of $x_i$ is represented in $\fun(BG,
\perf(k))$ by the smash product of $k$ and a finite $G$-complex with isotropy in a maximal
proper subgroup of $G$. As a result, the smash product 
$k^{hG}/(x_1, \dots, x_n)$
of the cofibers of
each of the $x_i$'s has to belong to the thick subcategory generated by $G_+$.
The converse direction is easier: the thick subcategory generated by any
nonzero $X \in \fun(BG, \perf(k))$ contains $k \wedge G_+$; one sees this from
the projection formula and the fact that $k \wedge G_+$ belongs to the thick
subcategory generated by the unit.  
\end{example} 

Every compact object in the stable module $\infty$-category 
can be represented by a finitely generated $k[G]$-module. A general result that
includes this appears in \cite{compactstable}. 
Any invertible object in $\st{G}$ is compact (as the unit is), and therefore
one can identify the group of stable equivalence classes of
(finite-dimensional) endotrivial modules in the sense of
\Cref{endotrivdef} and the Picard group of the $\e{\infty}$-ring
$k^{tG}$.

\section{Torus actions on $\infty$-categories}
\subsection{Generalities}
We start by reviewing some generalities about group actions. 

\begin{definition} 
Let $\mathcal{V}$ be an $\infty$-category and let $G$ be a topological group.
A \emph{$G$-action} on an object
$x \in \mathcal{V}$ will mean a functor of $\infty$-categories $BG \to
\mathcal{V}$ which maps the basepoint of $BG$ to $x$. We will call $\fun(BG, \mathcal{V})$ the \emph{$\infty$-category
of objects in $\mathcal{V}$ equipped with a $G$-action.} 
Recall that if $\mathcal{V}$ admits limits, the functor
\[ \mathcal{V} \to \fun(BG, \mathcal{V}),  \]
which gives an object of $\mathcal{V}$ the trivial $G$-action, admits a right
adjoint of homotopy fixed points which we will denote $()^{hG}$.

For example, we can
consider a
$G$-action on an $\infty$-category by taking $\mathcal{V}$ to be the $\infty$-category
$\mathrm{Cat}_\infty$ of $\infty$-categories. 
\end{definition} 

\begin{example}
\label{actionson1category}
Let $\mathcal{C}$ be a 1-category. To give an $S^1$-action on $\mathcal{C}$
is equivalent to giving an automorphism of the identity functor of
$\mathcal{C}$.
This follows from considering $S^1 = B \mathbb{Z}$ as a monoidal category
with one object whose endomorphisms are $\mathbb{Z}$: to
give an $S^1$-action on $\mathcal{C}$ is to give a 
multiplication map $B \mathbb{Z} \times \mathcal{C} \to \mathcal{C}$
satisfying natural identities.

Given an $S^1$-action on $\mathcal{C}$, 
the homotopy fixed points $\mathcal{C}^{hS^1}$ are the full subcategory of
$\mathcal{C}$ spanned by those objects on which this automorphism is the
identity. We note that actions of the \emph{monoid} $\mathbb{Z}_{\geq 0}$ were
considered in \cite{Dri} under the name \emph{$\mathbb{Z}_+$-category.}
\end{example}

\begin{example} 
\label{actonstablemodulecat}
Given a finite group $G$ and a central element $g \in G$, we obtain an
$S^1$-action on the category of (discrete) $k[G]$-modules given by multiplying
by $g$. 
\end{example} 

\begin{example} 
Let $R$ be a smooth commutative algebra over a field of characteristic zero and
let $r \in R$ be a nonzerodivisor.  Consider the
$\infty$-category $\perf(R)$ of all perfect $R$-module spectra. 
The subcategory $\perf(R)_{r-\mathrm{tors}} \subset \md(R)$ of all $r$-power
torsion perfect $R$-module spectra has an $S^1$-action, which gives an $r$-power
torsion perfect $R$-module spectrum $M$ the automorphism of multiplication by $1 + r$. 
In this case, we have an equivalence of $\infty$-categories
\[ \left(\perf(R)_{r-\mathrm{tors}}\right)^{hS^1} \simeq \mathrm{Coh}(R/r),  \]
where $\mathrm{Coh}(R/r) \subset \md(R/r)$ is the subcategory of those
$R/r$-module spectra that are perfect as $R$-modules,
by a theorem of Teleman. We refer to \cite[\S 3]{Preygelthesis} for a detailed  treatment.  
\end{example}


\subsection{The action on the stable $\infty$-category}
We now construct the analog of this action on the
stable module $\infty$-category of an abelian $p$-group. 
It is possible to use \Cref{actonstablemodulecat} together with
$\infty$-categorical localization, but it is convenient for our purposes to take a
different approach. 

\begin{construction}
\label{torusact}
Let $A$ be an abelian $p$-group of rank $n$, i.e., $A$ can be minimally
generated by $n$ elements. 
Then we can find a free abelian group $L$ of rank $n$ and a surjection
\( L \twoheadrightarrow A,  \)
with kernel $L' \subset L$.
Since $A$ and $L$ are abelian, we can model $BL$ and $BA$ by topological
abelian groups, so that
\( BL \to BA  \)
is a morphism of topological abelian groups. 
We thus obtain an action of $BL$ on $BA$ via this homomorphism and the
$BA$-action on itself by left translation. 

We have a fiber sequence
of spaces
\[ BL \to BA \to B^2 L',  \]
exhibiting $B^2 L'$ as the homotopy orbits $(BA)_{h BL}$. Note that $BL$ is
a torus $\mathbb{T}^n$ and $B^2 L'$ is $(\mathbb{CP}^\infty)^n$ (noncanonically).
Taking cochains with $k$-valued coefficients, we obtain a morphism of
$\e{\infty}$-rings
\[ F( B^2L'_+, k) \simeq k^{h BL'} \to k^{hA} \simeq F(BA_+, k),  \]
and $BL$ acts on $k^{hA}$ in the $\infty$-category of $\e{\infty}$-$k^{ h BL'}$-algebras such that the natural map 
$k^{h B L'} \to (k^{hA})^{h BL}$ is an equivalence. 
\end{construction}

This action has a Galois property, which enables one to determine the
$\infty$-category of $BL$-homotopy fixed points.
We discuss this below.

\subsection{Rognes's Galois theory and descent}
Let $G$ be a topological group with the homotopy type of  a finite complex. 
We will need the notion of a faithful $G$-Galois extension of
$\e{\infty}$-rings.

\begin{definition}[{Rognes \cite[Def. 4.1.3]{rognes}}] 
\label{galext}
Consider an extension $R \to R'$ of $\e{\infty}$-rings and suppose $G$ acts on
$R'$ in $\e{\infty}$-$R$-algebras. The extension is said to be a
\textbf{faithful $G$-Galois extension} if:
\begin{enumerate}
\item  $ R \to (R')^{hG}$ is an equivalence. 
\item  
The natural map $R' \otimes_R R' \to F(G_+, R')$ is an equivalence. 
\item Given an $R$-module $M$ such that $R' \otimes_R M$ is contractible, $M$
is itself contractible.
\end{enumerate}

\end{definition}

\begin{proposition} 
Suppose $R \to R'$ is 
a morphism of $\e{\infty}$-rings and suppose $R'$ is given a $G$-action in
$\e{\infty}$-$R$-algebras. Then $R \to R'$ is a faithful $G$-Galois extension
if and only if 
there exists an $\e{\infty}$-$R$-algebra $\widetilde{R}$ such that:
\begin{enumerate}
\item  
The
$\e{\infty}$-$\widetilde{R}$-algebra $R' \otimes_R \widetilde{R}$ is
$G$-equivariantly isomorphic to $F(G_+, \widetilde{R})$ (where the latter has a
$G$-action by translation)
\item
$\widetilde{R}$ is descendable
(cf. \cite[\S 3-4]{galois}, \Cref{def:desc} below) as an $\e{\infty}$-$R$-algebra.
\end{enumerate}
In particular, $G$-Galois extensions are stable under base-change. 
\end{proposition} 
\begin{proof} 
Suppose there exists $\widetilde{R}$ as above, so that we need to show that $R
\to R'$ is a faithful $G$-Galois extension. We first need to check that  $ R \to (R')^{hG}$ is an equivalence. 
We claim more strongly that if $\left\{X_n\right\}$ is the standard tower of
$R$-module spectra converging  to $(R')^{hG}$ (based on the simplicial
filtration of $BG$) then the 
cofiber of the natural map of towers 
$\left\{R\right\} \to \left\{X_n\right\}$ is nilpotent (cf. \cite[Sec.
3.1]{Mthick}).
To check this, we may make a base-change along the descendable morphism $R \to
\widetilde{R}$. 
In fact, descent along $R \to \widetilde{R}$ allows to reduce to the case $R' \simeq
F(G_{+}, R)$ to begin with, and in this case the cosimplicial object computing
$(R')^{hG}$
is split augmented over $R$. 

Next, the natural map $\widetilde{R} \otimes_R \widetilde{R} \to
F(G_+, \widetilde{R})$ is an equivalence as that can be checked after
base-change along $R \to R'$. The final condition (of \emph{faithfulness})  in \Cref{galext} can similarly be
checked by descent.

Conversely, if $R \to R'$ is a faithful $G$-Galois extension, then $R'$ is a dualizable $R$-module
\cite[Prop. 6.2.1]{rognes} and the third condition of faithfulness then implies that $R'$
is itself descendable \cite[Th. 3.36]{galois} so that the conditions 
 above are satisfied with $\widetilde{R} = R'$. 
\end{proof}

We now state the Galois descent theorem. 
For $G$ finite, this also appears in \cite{meier, GL}.
\begin{proposition}[{cf. \cite[\S 9]{galois}}]
\label{galdesc}
Given a faithful $G$-Galois extension $R \to R'$, one obtains a $G$-action
on the symmetric monoidal $\infty$-category $\md(R')$ of $R'$-modules,
and the
natural functor
\[ \md(R) \to \md(R')^{hG}  \]
is an equivalence of symmetric monoidal $\infty$-categories.
\end{proposition}
We return to our setup. 
\begin{proposition}[{cf. \cite[\S 9]{galois}}] 
\label{thismapisgalois}
Notation as in \Cref{torusact}, the natural map $k^{h BL'} \to k^{hA}$ exhibits the
target as a faithful $BL$-Galois extension of the source. 
\end{proposition} 
\begin{proof} 
In fact, one checks by an explicit calculation on homotopy groups that the map $k^{hBL'} \to k^{hA}$
makes the target into a free module (with shifts) over the source, e.g.,
$F(BC_{p+}, k) = k^{hC_p}$ is free over $F(BS^1, k) \simeq k^{hS^1}$. In
particular, it is descendable.
Moreover, we have a fiber square of spaces
\[ \xymatrix{
 BL \times BA \ar[r]\ar[d]  &  BA \ar[d] &  \\
 BA \ar[r] &  B^2 L'
},\]
i.e., $BA \to B^2 L'$ is a $BL$-torsor. 
By the convergence of the Eilenberg-Moore spectral sequence (cf. \cite[\S
1.1]{DAG13}), we have an
equivalence of $\e{\infty}$-$k^{hA}$-algebras,
\[ F(BL_+, k^{hA}) \simeq F( (BL \times BA)_+, k) \simeq k^{hA}
\otimes_{k^{hBL'}} k^{hA}, \]
which is equivariant for the $BL$-action, by translation on the
left-hand-side and on the second factor on the right-hand-side. 
This implies the claim. 
\end{proof}

As a result of \Cref{galdesc} and \Cref{thismapisgalois}, the canonical functor
\[ \md( F( (\mathbb{CP}^\infty)^n_+, k)) \simeq \md( k^{h BL'}) \to  \md( k^{hA})^{h BL} \]
is an equivalence of symmetric monoidal $\infty$-categories. Moreover, this
holds after arbitrary base-change along the source.

We will now extend this to the stable module $\infty$-category.
\begin{definition} 
Let $\mathbb{T}^n$ be the standard $n$-torus. 
Recall that $\pi_* k^{h \mathbb{T}^n} \simeq k[x_1, \dots, x_n]$ where, for
each $i$, 
$|x_i| = -2$.
We will let $k^{t \mathbb{T}^n}$ denote the localization of $
k^{h \mathbb{T}^n} = F(B\mathbb{T}^n_+, k)$ away from the iterated cofiber $F(B\mathbb{T}^n_+,
k)/(x_1, \dots, x_n)$ (cf. \cite{GM}), or equivalently the
$A^{-1}$-localization in $\md( k^{h \mathbb{T}^n})$ for $A = k^{h
\mathbb{T}^n}/(x_1 , \dots, x_n) \in \mathrm{Alg}( \md( k^{h \mathbb{T}^n}))$
(cf. \cite[\S 3]{MNN15}). 
If $M$ is a finitely generated free $\mathbb{Z}$-module, we will write $k^{t
BM}$ for the above construction where we identify $BM$ with an $n$-torus (in
any manner). 
\end{definition} 

We have
a pushout square of $\e{\infty}$-rings
\[  \xymatrix{
k^{h BL'} \ar[d] \ar[r] & k^{t BL'} \ar[d] \\
k^{h A} \ar[r] &  k^{t A}
},\]
because both the Tate constructions are obtained by the equivalent finite
localization (cf. \Cref{tateex}).

\begin{construction}
\label{galoistate1}
As a localization of $k^{hA}$, the Tate construction $k^{t A}$ inherits a
$BL $-action, and $k^{t BL'} \to k^{tA}$ is a
faithful $BL$-Galois extension. 
The action of $BL$ on the $\e{\infty}$-ring $k^{tA}$
induces a $BL$-action  on the symmetric monoidal $\infty$-category
$\st{A} \simeq \md( k^{tA})$. Note also that the $BL$-action extends to a
$BA$-action (which originates with the $BA$-action on itself by left
translation).
\end{construction}

We summarize the above analysis in the following. 

\begin{theorem}[{cf. \cite[\S 9]{galois}}] 
\label{basicdescresult}
Let $A$ be an abelian $p$-group of rank $n$, and choose a short exact sequence
$0 \to L' \to L \to A \to 0$ where $L \simeq \mathbb{Z}^n$.
Then one obtains an action of the torus $BL \simeq \mathbb{T}^n$ 
on the symmetric monoidal, $k$-linear $\infty$-category $\st{A}$ such that:
\begin{enumerate}
\item The homotopy fixed points $\st{A}^{h BL}$ are canonically identified with the
symmetric monoidal $\infty$-category $\md( k^{t
BL'})$.
\item The action of $BL$ extends to an action of $BA$ on $\st{A}$.  
\end{enumerate}
\end{theorem} 

\section{Torus actions on spectra}

In this section, we will describe some tools for working with 
torus actions on spectra. 
We note that most of the subtleties encountered here disappear after $2$ is
inverted. 
\subsection{$\e{\infty}$-rings and $\mathfrak{gl}_1$}

Let $R$ be an $\e{\infty}$-ring. Recall the spectrum of
\emph{units} $\mathfrak{gl}_1(R)$ \cite{MQRT, ABGHR1}. 
We have a natural equivalence of connected spaces
\[ (\Omega^\infty R)_{\geq 1} \simeq (\Omega^\infty \mathfrak{gl}_1(R))_{\geq 1}   \]
and in particular 
natural isomorphisms of homotopy groups $$\psi \colon \pi_i(R) \simeq \pi_i(\mathfrak{gl}_1(R))$$
for $i > 0$.
However, these isomorphisms are generally not compatible with the $\pi_*(S^0)$-module
structure.

\begin{proposition} 
\label{multbyhopfone}
Let $n \in \{ 1, 3, 7\}$, so that there exists an element $\delta_n \in \pi_n(S^0)$
of Hopf invariant one. 
Then if $R$ is any $\e{\infty}$-ring and if $x \in \pi_n(R)$, we
have
\begin{equation} 
\label{psidelta}
\delta_n \psi(x)  =  \psi(\delta_n x + x^2) \in \pi_{2n}( \gl_1(R)).
\end{equation} 
\end{proposition} 
\begin{proof} 
By naturality, it suffices to prove this in the  case where $R$ is the free
$\e{\infty}$-ring on $S^n$, so that 
\[ R \simeq S^0 \vee S^n \vee (S^n \wedge S^n)_{h\Sigma_2} \vee \dots,  \]
and $x \in \pi_n(R)$ is the tautological element. 
In this case, $\pi_{2n}(R) \simeq \pi_{2n}(S^0) \oplus \pi_{2n}(S^n) \oplus
\mathbb{Z}/2$.
It follows that we have a 
``universal''
formula
\[ \delta_n\psi(  x) =  \psi(a_1 x + a_2 + a_3 x^2), \]
where $a_1 \in \pi_n(S^0), a_2 \in \pi_{2n}(S^0), a_3 \in \mathbb{Z}/2$. 
It remains to determine the values of these coefficients  by considering
several specific examples. 

First, we claim that $a_2 = 0$. This follows because we have a map of
$\e{\infty}$-rings $R \to S^0$
that annihilates $x$, so that $\delta_n \psi(  x)$ must belong to the kernel of
the induced map. 
Second, we claim that $a_1 = \delta_n$; this follows by considering the case of
the square-zero $\e{\infty}$-ring $S^0 \vee S^n$, so that we have a canonical
equivalence of \emph{spectra} $\mathfrak{gl}_1(S^0 \vee S^n) \simeq
\mathfrak{gl}_1(S^0) \vee S^n$ by \cite[Prop. 6.5.3]{MS}.

Finally, we claim that $a_3 \neq 0$. 
For this, we consider the $\e{\infty}$-ring $R = \tau_{\leq 2n}\mathbb{F}_2\{x_n\}$
obtained by truncating the free $\e{\infty}$-algebra over
$\mathbb{F}_2$  on a degree $n$ class to degrees $[0, 2n]$, so that $\pi_*(R)
\simeq \mathbb{F}_2[x_n]/(x_n^3)$. 
It follows \cite[Prop. 5.2.2]{MS} that the $k$-invariant of $\mathfrak{gl}_1(R)$ is 
the map $H \mathbb{F}_2[n] \to H \mathbb{F}_2[2n+1]$, given by
$\mathrm{Sq}^{n+1}$. 
However, this implies that the classes in degree $n$ and $2n$ of
$\pi_*\mathfrak{gl}_1(R)$ are connected by a $\delta_n$, as desired. 
\end{proof} 
\subsection{The group algebra}
Let $X$ be a spectrum with  an $S^1$-action. Equivalently, $X$ is a module
spectrum over the topological group ring $\Sigma^\infty_+ S^1$. 
Since $S^1$ is a commutative topological group, the group ring  
$\Sigma^\infty_+ S^1$ has the structure of an $\e{\infty}$-ring, and we obtain an action of
$\pi_*( \Sigma^\infty_+ S^1)$ on $\pi_*(X)$. 
We will need the following description of the former, which is well-known to experts. 
\begin{proposition} 
The homotopy groups of the ring spectrum $\Sigma^\infty_+ S^1$ are given by
the algebra
$\pi_*(S^0)[y]/(y^2 + y \eta)$ where $|y| = 1$.
More generally, we have an equivalence of algebras
\begin{equation} \label{pitorusalg}\pi_*( \Sigma^\infty_+ \mathbb{T}^n) \simeq \pi_*(S^0) [y_1,\dots,
y_n]/(y_i^2 = \eta y_i), \quad |y_i| = 1.  \end{equation}
\end{proposition} 
\begin{proof} 
As a spectrum, $\Sigma^\infty_+ S^1 \simeq S^0 \vee S^1$, so it suffices to
determine the multiplicative structure. 
We consider the map of $\e{\infty}$-ring spectra
$ \Sigma^\infty_+ S^1 \to S^0  $
obtained from the map of groups $S^1 \to 1$. Let $y \in \pi_1(
\Sigma^\infty_+ S^1)$ denote a generator of the kernel of $\pi_1(
\Sigma^\infty_+ S^1) \to \pi_1( S^0)$. 

Using the universal property of $\gl_1$ \cite[Th. 5.2]{ABGHR1}, we see that 
there exists a map
\[ H\mathbb{Z}[1] \to \mathfrak{gl}_1( \Sigma^\infty_+ S^1),  \]
whose image in homotopy is generated by $\psi( y)$. 
This map is adjoint to the identity map
\[ \Sigma^\infty _+\Omega^\infty (H \mathbb{Z}[1]) \simeq \Sigma^\infty_+ S^1
\to \Sigma^\infty_+ S^1. \]
In particular, we find that $\eta \psi(y) =0 $ in $\pi_{2}( \gl_1(
\Sigma^\infty_+ S^1))$. 
It follows from \eqref{psidelta} that
\[ \psi( \eta y) = \psi( y^2),  \]
so that $y^2 = \eta y$ as $\psi$ is an isomorphism.
\end{proof}

\begin{remark} 
We can also see this as follows. 
The ring structure on $\Sigma^\infty_+ S^1$ comes from the multiplication map
$m \colon S^1 \times S^1 \to S^1$. On unreduced suspensions, we obtain a map
\[ \Sigma^\infty_+ m \colon \Sigma^\infty_+ (S^1 \times S^1) \simeq S^0 \vee S^1
\vee S^1 \vee S^2 \to \Sigma^\infty_+ S^1 \simeq S^0 \vee S^1.  \]
However, under this identification
the stable map $S^2 \to S^1$ that  one obtains is the stabilization of the unstable Hopf
fibration $S^3 \to S^2$. This determines the multiplicative structure on
$\pi_*( \Sigma^\infty_+ S^1)$ as
claimed. 
We are grateful to Mike Hill for this remark. 
\end{remark}

\begin{remark} 
\label{multbyy1}
Let $X$ be a spectrum with an $S^1$-action.
We can describe multiplication by $y$ in the following manner on $\pi_0$. 
Consider the $H$-space $\Omega^\infty X$ and let $\ast \to \Omega^\infty X$ be a map of
\emph{unpointed} spaces corresponding to an element in $\pi_0 X$. The
$S^1$-action on $\Omega^\infty X$ then extends this to an unpointed map $S^1 \to \Omega^\infty X$, which
we can (using the $H$-space structure on $\Omega^\infty X$) turn into a \emph{pointed} map
$S^1 \to \Omega^\infty X$, i.e., a class in $\pi_1(X)$. This gives the map 
\[ \pi_0(X) \to \pi_1(X)  \]
of multiplication by $y$. 

It follows from this that if $X$ is a spectrum with $S^1$-action, the
multiplication by $y$ on $\pi_0$ is determined by the $S^1$-action on
$\Omega^\infty X$ together with its $H$-space structure. 
As a result, the multiplication by $y$ on $\pi_i, i > 0$ 
is determined by the action on $\Omega^\infty X$ to begin with. 
\end{remark}

\subsection{The spectral sequence}
In this subsection, we review the spectral sequence for taking homotopy fixed
points of $\mathbb{T}^n$-actions.
This spectral sequence
is essentially the one described in \cite[\S 1.4]{ptypical} (where it starts
from $E_1$ instead).

\begin{construction}
We can make $S^0$ into a $\Sigma^\infty_+ \mathbb{T}^n$-module,
or a spectrum with $\mathbb{T}^n$-action, by making $\mathbb{T}^n$ act
trivially.
Let $X$ be any spectrum with $\mathbb{T}^n$-action. 
Then we have
an equivalence of spectra
\[ X^{h\mathbb{T}^n} \simeq \hom_{\Sigma^\infty_+ \mathbb{T}^n}(S^0, X),  \]
and, thanks to the calculation \eqref{pitorusalg},  we obtain a spectral sequence \cite[Ch. IV, Th. 4.1]{EKMM}
\begin{equation}\label{S1ss} E_2^{s,t} = \mathrm{Ext}^{s,t}_{\pi_*(S^0)[y_1,
\dots, y_n]/(y_i^2 + y_i \eta)}( \pi_*(S^0), \pi_*(X))
\implies \pi_{t-s} (X^{h\mathbb{T}^n}).  \end{equation}
Here the module structure of $\pi_*(S^0)$  is such that each $y_i$ acts trivially.
The differentials run $d_r \colon E_r^{s,t} \to E_r^{s +
r , t + r-1}$.
\end{construction}

\begin{construction}
Suppose $\eta$ acts by zero on $\pi_*(X)$. 
In this case, by using natural $\mathrm{Ext}$ adjunctions in both variables, we can rewrite the
$E_2$ page 
of \eqref{S1ss}
as 
\begin{equation}\label{newE2ss} E_2^{s,t} \simeq
\mathrm{Ext}_{(\pi_*(S^0)/\eta)[y_1, \dots, y_n]/(y_i^2)}( 
\pi_*(S^0)/\eta, \pi_*(X))
\simeq \mathrm{Ext}_{\mathbb{Z}[y_1, \dots, y_n]/(y_i^2)}( \mathbb{Z}, \pi_*(X)).
\end{equation}
Here we use the inclusion $\mathbb{Z}[y_1, \dots, y_n]/(y_i^2) \subset \pi_*(\Sigma^\infty_+
\mathbb{T}^n)/\eta$ and we make the $y_i$ act by zero on $\mathbb{Z}$. 

\end{construction}

We will need a condition that ensures that this spectral sequence degenerates
at $E_2$. 
Let $A = \bigotimes_{i=1}^n \mathbb{Z}[y_i]/y_i^2$. 
\begin{proposition} 
\label{noext}
If $M$ is an abelian group, consider $M \otimes_{\mathbb{Z}} A$ as an
$A$-module by extending scalars. Then we have
\[ \mathrm{Ext}^s_A(  \mathbb{Z}, M \otimes_{\mathbb{Z}} A)  = 
0 , \quad s > 0. 
\]
\end{proposition} 
\begin{proof}
This follows because there is an isomorphism of $A$-modules
\[ M \otimes_{\mathbb{Z}} A \simeq \hom_{\mathbb{Z}}(A, M),  \]
and now one can use $\mathrm{Ext}$ adjunctions to obtain
\[ \mathrm{Ext}^s_A( \mathbb{Z}, M \otimes_{\mathbb{Z}} A ) 
\simeq \mathrm{Ext}^s_A( \mathbb{Z}, \hom_{\mathbb{Z}}(A, M)) \simeq
\mathrm{Ext}^s_{\mathbb{Z}}(\mathbb{Z}, M) = 0 , \quad s> 0.
\]
\end{proof} 

\begin{definition} 
Let $X$ be a spectrum with a $\mathbb{T}^n$-action. Suppose $\eta$ acts by zero
on $\pi_*(X)$. 
We will say that $\pi_*(X)$ is \textbf{relatively projective} if, as a module
over $A = \mathbb{Z}[y_1, y_2, \dots, y_n]/(y_i^2) \subset \pi_*(\Sigma^\infty_+
\mathbb{T}^n)/\eta$, it is  a retract of a  direct sum of modules that are
obtained from 
$\mathbb{Z}$-modules by extending scalars.
\end{definition}

\begin{example} 
\label{coinducedisfree}
Let $Y$ be a spectrum such that $\eta$ acts by zero on $\pi_*(Y)$. 
Then we can form the coinduced spectrum $X = F( \mathbb{T}^n_+, Y)$, which inherits
a $\mathbb{T}^n$-action from the first factor. The homotopy groups are 
relatively projective. In fact, as spectra with $\mathbb{T}^n$-actions, we
have an equivalence $X \simeq \Sigma^{-n} \Sigma^\infty_+
\mathbb{T}^n \wedge Y$ by Atiyah duality; it is clear that the latter has
relatively projective homotopy groups. 
\end{example}

\begin{corollary} 
\label{relativefreehfp}
Suppose $X$ is a spectrum with $\mathbb{T}^n$-action. Suppose that $\eta$ acts
trivially on $\pi_*(X)$ and that $\pi_*(X)$ is relatively projective. Then the map 
$\pi_*(X^{h\mathbb{T}^n}) \to \pi_*(X)$ is an injection whose image is given by
$\bigcap_{i=1}^n
\mathrm{ker}(y_i)$.
\end{corollary} 
\begin{proof} 
By \Cref{noext} above,  one sees
that the spectral sequence degenerates after $E_2$ with the desired outcome. 
\end{proof} 

We now consider a few different examples of the spectral sequence.

\begin{example} 
\label{discreteex}
Let $X$ be a discrete spectrum (i.e., $\pi_i X  = 0$ if $i \neq 0$) with a
$\mathbb{T}^n$-action, necessarily trivial. 
Then one finds that the above spectral sequence is given additively at $E_2$ by $\pi_0 X
\otimes \mathbb{Z}[\xi_1, \dots, \xi_n]$ where each $\xi_i$ is in bidegree $(1, -1)$.  
The spectral sequence collapses at $E_2$ and one obtains the cohomology of
product of copies of a
projective space. 
\end{example}

\begin{corollary} 
\label{evendegrees}
Suppose $X$ is a spectrum with a $\mathbb{T}^n$-action
and suppose that $\pi_*(X)$ is concentrated in even degrees. 
Then $\pi_*(X^{h\mathbb{T}^n})$ is concentrated in even degrees.
\end{corollary} 
\begin{proof} 
By induction, we may take $n = 1$ and $\mathbb{T}^n = S^1$. 
In this case, we see that $\pi_*(X)$ is, as a module over $\mathbb{Z}[y]/y^2$,
a sum of modules $M_{2i}$ concentrated in a single degree $2i$ (for $i \in
\mathbb{Z}$). The spectral sequence 
\eqref{newE2ss} is easily seen to degenerate for reasons of degree and give the result. 
Alternatively, the result follows by filtering $X$ via the Postnikov filtration. 
\end{proof} 

\begin{corollary} 
\label{categoryS1}
Let $X$ be a spectrum concentrated in degrees $0$ and $1$ with an
$\mathbb{T}^n$-action; we do not assume $\eta$ acts by zero. Suppose given $x \in \pi_0 X$. Then the following are
equivalent: 
\begin{enumerate}
\item $x$ lifts to an element of $\pi_0 X^{h \mathbb{T}^n}$. 
\item $x$ is annihilated by the operators $y_i, i = 1, \dots, n$.
\end{enumerate}
\end{corollary}
\begin{proof}
This follows from the spectral sequence \eqref{newE2ss}. In particular, if $x$
is annihilated by the operators $y_i$, it defines a class in the 0-line $\hom_{\pi_*(
\Sigma^\infty_+ \mathbb{T}^n)}( \pi_*(S^0), \pi_*(X))$.  We need to show that it survives the
spectral sequence, i.e., it supports no differentials. 
However, we observe that 
\[ \mathrm{Ext}^{s, s-1}_{\pi_*( \Sigma^\infty_+ \mathbb{T}^n)}( \pi_*(S^0),
\pi_*(X))  = 0 \]
for $s \geq 2$ for grading reasons, cf. \Cref{discreteex}. 
This should be compared to \Cref{actionson1category} via the identification
between spectra concentrated in degree $[0, 1]$ and symmetric monoidal
groupoids where every object is invertible. 
\end{proof}

\section{Proof of Dade's theorem}

\subsection{Picard spectra}
Given a symmetric monoidal $\infty$-category $\mathcal{C}$, we will let
$\pics(\mathcal{C})$ denote the (connective) \emph{Picard spectrum} of
$\mathcal{C}$, cf., e.g.,  \cite[\S 2.2]{MS} for a discussion. Recall that: 
\begin{enumerate}
\item  $\pi_0 \pics(\mathcal{C}) =
\pic(\mathcal{C})$ is the Picard group of
isomorphism classes of invertible objects in $\mathcal{C}$.
\item  $\pi_1 \pics(\mathcal{C})$ is the group of
homotopy classes of self-equivalences of the unit $\mathbf{1} \in
\mathcal{C}$.
\item 
$\pi_i \pics(\mathcal{C})$ for $i \geq 2$ is the $(i-1)$st
homotopy group of the endomorphism space of the unit. 
\end{enumerate} 

Given an $\e{\infty}$-ring $R$, we will also write $\pics(R)$ for $\pics(
\md(R))$. We have a natural equivalence of spectra $\tau_{\geq 1} \pics(R) \simeq \Sigma \gl_1(R)$. 

Let $\mathrm{Cat}_{\otimes}$ be the $\infty$-category of symmetric monoidal
$\infty$-categories and symmetric monoidal functors between them, and let
$\sp_{\geq 0}$ denote the $\infty$-category of connective spectra. We obtain a
functor
\[ \pics\colon \mathrm{Cat}_{\otimes} \to \sp_{\geq 0} , \]
which {commutes with limits} \cite[Prop. 2.2.3]{MS}. 

For instance, it follows from 
\Cref{basicdescresult} (with the notation from there)
that we have equivalences of connective spectra
\begin{equation} \label{picshTn} \pics( k^{t BL'}) \simeq \tau_{\geq 0}( \pics(
k^{t A})^{h
B L})
\simeq \tau_{\geq 0}( \pics(
k^{t A})^{h
\mathbb{T}^n })
\simeq \tau_{\geq 0} (\pics( \st{A})^{h \mathbb{T}^n})
.  \end{equation}
To prove Dade's theorem, we will calculate the left-hand-side of this
equivalence, and show that any element of the Picard group of $\st{A}$
survives to the homotopy fixed points 
$\pics( \st{A})^{h \mathbb{T}^n}$. 
\subsection{The general tool}

In this subsection, we prove our main general tool 
for Picard groups (\Cref{descpictool}), which gives a criterion for when
invertible modules can be descended along a faithful $\mathbb{T}^n$-Galois
extension. 

\begin{proposition} 
Let $R \to R'$ be a faithful $\mathbb{T}^n$-Galois  extension. Suppose
$\pi_*(R) \to \pi_*(R')$ exhibits the target as a free module over the source and that $\eta = 0$ in $\pi_*(R)$. 
Then $\pi_*(R')$ is relatively projective 
as a module
over $\mathbb{Z}[y_1, \dots, y_n]/(y_i^2) \subset \pi_*( \Sigma^\infty_+
\mathbb{T}^n)/\eta$.
\end{proposition} 
\begin{proof} 
The assumption of freeness of $\pi_*(R')$ as $\pi_*(R)$-module imply that it suffices to check the claim about relative projectivity after 
base-change along $R \to R'$, i.e., to show that the action on 
$\mathbb{Z}[y_1, \dots, y_n]/(y_i^2)$ on $\pi_*( R' \otimes_R R')$ makes the
target relatively projective. However, this follows again because we have a
$\mathbb{T}^n$-equivariant equivalence
$R' \otimes_R R' \simeq F(\mathbb{T}^n_+, R')$. 
\end{proof} 

We will now need a general result about homotopy fixed points. 
\begin{proposition} 
\label{hfpunits}
Let $R \to R'$ be a faithful $\mathbb{T}^n$-Galois extension of
$\e{\infty}$-rings. Suppose that: 
\begin{enumerate}
\item $\eta$ acts trivially on $\pi_*(R)$.
\item  
Every element in
$\pi_1 R $ squares to zero.
\item
$\pi_*(R')$ is a free module over $\pi_*(R)$. 
\end{enumerate}
Then 
we have: 
\begin{equation}
\label{noeven}
\pi_i (\tau_{\geq 1} \gl_1(R')^{h\mathbb{T}^n})  = 
\begin{cases} 
\pi_i \gl_1(R) & i \geq 1 \\
0 & i \leq 0 \text{ and } i \ \text{even} \\
?? & i \leq 0 \text{ and } i \ \text{odd}
 \end{cases}.\end{equation} 
 If we do not assume hypothesis (2), the above conclusion \eqref{noeven} still holds
 after inverting 2. 
\end{proposition} 
Of course, the conclusion of the theorem is the vanishing of $\pi_i (\tau_{\geq
1} \gl_1(R')^{h\mathbb{T}^n})$ for $i \leq 0$ even, but we have stated in this
manner for ease of reference. 
\begin{proof} 
First, we consider the case where $n = 1$ so that $R \to R'$ is an $S^1$-Galois extension.
We claim first that $\eta$ acts by zero on the homotopy groups of 
$\tau_{\geq 1} \mathfrak{gl}_1(R)$. On $\pi_1$, this follows from \Cref{multbyhopfone}. 
On higher homotopy groups, this follows because we have an equivalence of
spectra
\[ R_{[i, i+1]} \simeq \gl_1( R_{[i, i+1]}), \quad i \geq 2  \]
by \cite[\S 5.1]{MS}. 
It follows that $\pi_*( \tau_{\geq 1}
(R))$ is  the direct sum of relatively projective
$\mathbb{Z}[y]/y^2$-modules and graded $\mathbb{Z}[y]/y^2$ modules $M$
concentrated in degree one. 
By \Cref{multbyy1} and our assumptions, 
the same holds for 
$\pi_*( \tau_{\geq 1}\gl_1
(R))$. 
The spectral sequence now implies that $\pi_i (\tau_{\geq 1} \gl_1(R')^{hS^1})$
is as desired. 

Now consider a torus of arbitrary rank $\mathbb{T}^n$. 
Choose a decomposition $\mathbb{T}^n \simeq \mathbb{T}^{n-1} \times S^1$. 
By induction on $n$, we can assume that $\pi_i ( \tau_{\geq 1}
\gl_1(R'))^{h\mathbb{T}^{n-1}} = 0$ for $i \leq 0$ even. 
We consider the $S^1$-equivariant cofiber 
sequence
\[ 
\tau_{\geq 1}\gl_1 (R'^{h \mathbb{T}^{n-1}}) \to 
(\tau_{\geq 1}
\gl_1(R'))^{h\mathbb{T}^{n-1}} \to C , 
\]
where by the inductive assumption the homotopy groups of $C$ are concentrated in negative,
odd degrees. 
Taking $S^1$-homotopy fixed points, we obtain a cofiber sequence of spectra
\[ 
\left(\tau_{\geq 1}\gl_1 (R'^{h \mathbb{T}^{n-1}})\right)^{hS^1} \to 
(\tau_{\geq 1}
\gl_1(R'))^{h\mathbb{T}^{n}} \to C^{hS^1} . 
\]
By \Cref{evendegrees}, it follows that $C^{hS^1}$ is concentrated in negative,
odd degrees. The map $R \to R'^{h\mathbb{T}^{n-1}}$ is a faithful $S^1$-Galois 
extension to which the hypotheses of this result applies, and the case of rank
one implies that 
$\left(\tau_{\geq 1}\gl_1 (R'^{h \mathbb{T}^{n-1}})\right)^{hS^1} $ has
vanishing $\pi_i$ for $i \leq 0$ even. Combining these completes the proof. 
When we invert $2$ on $\gl_1(R)$, we note that one does not have to worry about
$\eta$ and the above goes through more simply. 
\end{proof}

We will now consider the following general question. Let $R \to R'$ be a
faithful $\mathbb{T}^n$-Galois extension and let $\mathcal{L} \in \pic(R')$. We
ask when $\mathcal{L}$ descends to an $R$-module. 
Since there is a $\mathbb{T}^n$-action on the $\infty$-category $\md(R')$, each
element $a \in \pi_1(\mathbb{T}^n)$ induces a natural automorphism of the
identity functor of $\md(R)$. In particular, each $a$ induces a natural
automorphism of $R'$-modules $\mathcal{L} \simeq \mathcal{L}$, which is
classified by an element in $\pi_0(R')^{\times}$. A necessary condition for  
$\mathcal{L}$  to descend is that this monodromy automorphism should be the
identity. 

\begin{theorem} 
\label{descpictool}
Let $R \to R'$ be a faithful $\mathbb{T}^n$-Galois extension of $\e{\infty}$-ring
spectra. 
Suppose the hypotheses of \Cref{hfpunits} are satisfied. 
Let $\mathcal{L} \in \pic(R')$. Then $\mathcal{L}$ descends to an invertible
$R$-module if and only if for every $a \in \pi_1(\mathbb{T}^n)$, the induced
monodromy automorphism $a\colon \mathcal{L} \to \mathcal{L}$ is the
identity.

If we assume only that the first and third hypotheses of \Cref{hfpunits}
are satisfied,
then given any such $\mathcal{L}$, the tensor power $\mathcal{L}^{\otimes 2^n}$ descends for $n \gg 0$. 
\end{theorem} 
\begin{proof} 

By Galois descent, we need to show that the map 
\[ \pics( R')^{h\mathbb{T}^n} \to 
\pics(R') \to \tau_{\leq 1} \pics(R')
\]
has the property that the class of $\mathcal{L}$ belongs to the image in
$\pi_0$. 
For this, we consider the diagram of spectra
\[  \xymatrix{
\Sigma(\tau_{\geq 1} \gl_1( R'))^{h\mathbb{T}^n} 
\ar[r] &  
 (\pics( R'))^{h\mathbb{T}^n}  \ar[r]  \ar[rd] &  
(\tau_{\leq 1}\pics( R'))^{h\mathbb{T}^n} \ar[d] \\
& & \tau_{\leq 1}\pics( R'),
}
\]
where the top row is a cofiber sequence of spectra. 
The assumption about monodromy implies that $\mathcal{L}$ belongs to the image of the vertical
map of spectra, cf. \Cref{categoryS1}. It therefore suffices to show that 
the horizontal map
\[ (\pics( k^{t A}))^{h\mathbb{T}^n}  \to 
(\tau_{\leq 1}\pics( k^{t A}))^{h\mathbb{T}^n} 
\]
induces a surjection on $\pi_0$. This in turn follows because 
$\pi_{-2}( \tau_{\geq 1} \gl_1( R'))^{h\mathbb{T}^n}) = 0$
by \Cref{hfpunits}.
\end{proof} 

\subsection{Dade's theorem and its variants}
In this subsection, we complete the proof of Dade's theorem and its version
over an open subset of projective space. 
Our first goal is to calculate the Picard group of $k^{t \mathbb{T}^n}$ for any
$n$. 
\begin{proposition} 
\label{picstT}
The Picard group of $k^{t \mathbb{T}^n}$ is cyclic, generated by the
suspension $\Sigma k^{t \mathbb{T}^n}$. 
\end{proposition} 
\begin{proof} 
As explained in \cite[\S 9.4]{galois}, there exists an even periodic derived scheme
$\mathfrak{X}$ (with structure sheaf $\otop$) such that $\Gamma(
\mathfrak{X}, \otop) \simeq k^{t \mathbb{T}^n}$ and with the following
properties: 
\begin{enumerate}
\item The underlying ordinary scheme of $\mathfrak{X}$ is $\mathbb{P}^{n-1}_k$. 
\item $\pi_2\otop $ is the sheaf $\mathcal{O}(-1)$ on $\mathbb{P}^{n-1}_k$.
\item There is an equivalence of symmetric monoidal $\infty$-categories $\md(
k^{t \mathbb{T}^n}) \simeq \mathrm{QCoh}( \mathfrak{X})$. Equivalently,
$\mathrm{QCoh}(\mathfrak{X})$ admits the unit as a compact generator
\cite[Prop. 7.1.2.7]{HA}. 
\end{enumerate}

The derived scheme $\mathfrak{X}$ is constructed by inverting each of the
polynomial generators in $k^{h \mathbb{T}^n}$ to obtain an even periodic
$\e{\infty}$-ring with $\pi_0$ a polynomial algebra over $k$. One glues
together these affine spaces to form a $\mathbb{P}^{n-1}_k$. 
The third assertion follows from the ampleness of $\pi_{-2} (\otop)$, cf.
\cite[Prop 3.22]{MMaffine}.

Zariski locally, $\otop$ is a sheaf of even periodic  ring spectra with
\emph{regular} $\pi_0$, so that the Picard group is algebraic (cf.
\cite{BakerRichter},
\cite[Th. 2.4.4]{MS}). 
In particular, it follows that we can calculate the Picard group of 
$k^{t \mathbb{T}^n}$ using the 
descent spectral sequence \cite[\S 3]{MS}. Using the cohomology of line bundles
on
projective space \cite[III.5]{Ha77}, one sees that there is no contribution to
the Picard group.  We note that $H^i( \mathbb{P}^{n-1}, \mathcal{O}(-r)) = 0$ if $i
\notin \left\{0, n-1\right\}$ for all $r \in \mathbb{Z}$.
We have $H^{n-1}( \mathbb{P}^{n-1}, \mathcal{O}(-r))
\simeq H^{n-1}( \mathbb{P}^{n-1}, \pi_{2r}\otop)
\neq 0$ for $r \geq n$, but these will not contribute in the Picard spectral
sequence.
\end{proof}


\begin{theorem}[{Dade \cite[Th. 10.1]{DadeII}}] 
If $A$ is an abelian $p$-group and $k$ is a field of characteristic $p$, then
the Picard group of $\st{A}$ is cyclic, generated by the suspension of the unit.
\end{theorem} 
\begin{proof}
If $ A= C_2$, the homotopy groups of $k^{tA}$ 
are a graded field $k[u_1^{\pm 1}]$ and every module over $k^{tA}$ is free, so the Picard group
is trivial. 
We will therefore assume $A \neq C_2$.
Then, every element of $\pi_1 (k^{tA})$ squares to zero.

First, we use \Cref{stmodtate} to identify $\st{A}$ and $\md( k^{tA})$, so we
equivalently need to determine the Picard group of the $\e{\infty}$-ring
$k^{tA}$.
We have a faithful $\mathbb{T}^n$-Galois extension of $\e{\infty}$-rings $k^{t \mathbb{T}^n} \to k^{tA}$. 
Note that $\pi_*(k^{tA})$ is a free module over $\pi_*(k^{t \mathbb{T}^n})$; in
fact, one checks this for the faithful $\mathbb{T}^n$-Galois extension
$k^{h\mathbb{T}^n} \to k^{hA}$ of which this is a base-change. 

Fix now an invertible $k^{t A}$-module $M$. By \Cref{descpictool}, it descends
to $k^{t \mathbb{T}^n}$ if and only if the monodromy elements in $\pi_1
(\mathbb{T}^n)$ act trivially on $M$, i.e., give the identity in $(\pi_0
k^{tA})^{\times} = k^{\times}$. 
Since the $\mathbb{T}^n$-action on $\md(k^{tA})$ extends to a $B A$-action, the
monodromy necessarily acts by $p$-power torsion elements. Since $k^{\times}$ is
$p$-torsion free, the monodromy is trivial. 
It follows that $M$ descends to an invertible module over $k^{t \mathbb{T}^n}$. 
By \Cref{picstT}, the Picard group of $k^{t \mathbb{T}^n}$ is generated by the suspension
of the unit, so we are done. 
\end{proof}

We now describe how one can use the above arguments to obtain partial results
about the Picard groups of localized stable module categories. 
\begin{construction}
Let $U \subset \mathbb{P}^{n-1}_k$ be an open subset. 
We can then form an appropriate localization of the stable module category.
Namely, we construct the even periodic derived scheme  $\mathfrak{X}$ as in
\Cref{picstT} and consider the open subscheme $\mathfrak{U} \subset
\mathfrak{X}$ corresponding to $U$. 
We construct
the $\e{\infty}$-rings
\[ k^{t \mathbb{T}^n}_U \stackrel{\mathrm{def}}{=} 
\Gamma( \mathfrak{U}, \otop), \quad 
k^{tA}_U  \simeq k^{t \mathbb{T}^n}_U \otimes_{k^{t \mathbb{T}^n}} k^{tA}.  \]
The associated 
module $\infty$-category $\md( k^{tA}_U)$ is the localization (over $U$) of the
stable module $\infty$-category.
\end{construction}

In general, we do not know how to control the Picard group of these localized
stable module $\infty$-categories. However, we can obtain the following result.

\begin{proposition} 
Given an element $\mathcal{L} \in \pic( k^{tA}_U)$, $\mathcal{L}^{\otimes p^s}$
is a suspension of the unit for $s \gg 0$.
In particular, the quotient of $\pic((\st{A})_U)$ by the cyclic part generated
by the suspensions of the unit is $p$-power torsion.
\end{proposition} 
\begin{proof} 
We first argue that the Picard group of $k_U^{t \mathbb{T}^n}$ is cyclic 
after
$p^{-1}$-localization.
First, $\mathfrak{U}$ is also 0-affine (cf. \cite[Prop. 3.28]{MMaffine}),
i.e., one has an equivalence of symmetric monoidal $\infty$-categories $\md( k^{t\mathbb{T}^n}_U) \simeq \mathrm{QCoh}(
\mathfrak{U})$. Now the Picard group of the
\emph{scheme} $U$ is cyclic, generated by the twisting sheaf since the
restriction map $\mathrm{Pic}(\mathbb{P}^{n-1}_k) \to \mathrm{Pic}(U)$ is
surjective by regularity. The remaining terms in the descent spectral sequence
as in 
\Cref{picstT} 
that can contribute to the Picard group of $k_U^{t \mathbb{T}^n}$ (or
equivalently $\mathfrak{U}$) are all
$p$-power torsion and there are only finitely many of them.  

The descent from $k^{tA}_U$ to $k^{ t\mathbb{T}^n}_U$ can be carried out similarly. 
Namely, we claim that if $\mathcal{L} \in \pic(k^{tA}_U)$, there exists $r$ such that
$\mathcal{L}^{\otimes p^r}$ descends to $\pic( k^{t \mathbb{T}^n}_U)$. 
In fact, by \Cref{descpictool} it suffices to show that the monodromy action
on $\mathcal{L}^{\otimes p^r}$ (which gives an element of $\pi_0
(k^{tA}_U)^{\times}$) is trivial. Since this monodromy action is the $p^r$th
power of the monodromy element for $\mathcal{L}$ itself, it suffices to take
$r$ large enough that $p^r $ annihilates $A$. 
Note that if $p = 2$, there may be elements in $\pi_1$ which do not square to
zero, but we use instead the last 
claim of \Cref{descpictool}.
\end{proof}

\begin{theorem} 
Suppose $U \subset \mathbb{P}^{n-1}_k$ is affine and that $p > 2$.  
Let $\mathcal{L} \in \pic( (\st{A})_U)$. Suppose $\mathcal{L}$ is the
$U$-localization of a compact object in 
$\st{A}$. Then $\mathcal{L}$ is the suspension of a unit.
\end{theorem} 

\begin{proof} 
Suppose $\mathcal{L} \in \pic( (\st{A})_U)$. 
Each generator $a \in A$ induces an $S^1$-action on $(\st{A})_U$ given by
multiplication by $a$. 
In particular, multiplication by $a$ on $\mathcal{L}$ gives an element in 
$\pi_0 (k^{tA}_U)^{\times}$.
We claim that this element is trivial, or equivalently that  $a\colon \mathcal{L}
\to \mathcal{L}$ is homotopic to the identity. 
Since $\mathcal{L}$ is invertible in $(\st{A})_U$, it suffices to show that the
trace of the endomorphism $a - 1$ of $\mathcal{L}$ is equal to zero. 
However, we know that $\mathcal{L}$ arises as the $U$-localization of a
finite-dimensional $A$-representation $V$, which is dualizable in $\st{A}$. 
The trace of $a-1$ on $V$ is easily seen to be zero, so the trace of $a-1$ on
$\mathcal{L}$ is zero as well. 
This shows that every generator $a \in A$ acts as the identity on $\mathcal{L}$. 
By \Cref{descpictool}, this means that $\mathcal{L}$ descends to an
invertible module over $k^{t
\mathbb{T}^n}_U$. 
Since $k^{t \mathbb{T}^n}_U$ is regular, the Picard group is algebraic and the
result follows. 
\end{proof} 

Given a compact object in $(\st{A})_U$ (e.g., an  invertible one), we recall
that the obstruction to its being the $U$-localization of a compact object in
$\st{A}$ lives in $K_0$ (cf. \cite[5.2.2]{TT} for the analogous result for
extensions of perfect complexes over schemes). 

\section{Stratification via projective space}

It is known by the work of Benson-Iyengar-Krause that one can ``stratify'' objects of the stable module category
$\st{G}$ of
a $p$-group $G$ via the cohomology $H^{\mathrm{even}}(BG; k)$ and its
homogeneous prime ideals. In particular, one can classify the localizing
subcategories \cite{BIK}. 
More generally, one can carry this out for $\md( k^{hG})$. 

The key case is when $G$ is elementary abelian. 
In \cite{BIK}, this is proved by showing that the classification of localizing
subcategories of 
$\md( k^{hG})$ is equivalent to that of the $\infty$-category  of modules over a
ring spectrum whose homotopy groups are a \emph{polynomial ring.} 
Here the classification is much simpler. 

In this section, we use the comparison with the Tate construction for a torus
to give another approach to this reduction, and thus to their results.

\subsection{Generalities}

Let $\mathcal{C}$ be a presentable stable $\infty$-category. 
\begin{definition}
We recall that a \emph{localizing} subcategory $\mathcal{C}' \subset \mathcal{C}$ is a full
stable subcategory closed under arbitrary colimits. We will also assume that
any object of $\mathcal{C}$ equivalent to an object in $\mathcal{C}'$ is itself
in $\mathcal{C}'$.
We will let $\loc(\mathcal{C})$ denote the class
of all localizing subcategories of $\mathcal{C}$.

Given a full subcategory $\mathcal{D} \subset \mathcal{C}$, the intersection
of all localizing subcategories containing $\mathcal{D}$ is a localizing
subcategory, and is said to be the localizing subcategory \emph{generated} by
$\mathcal{D}$. \end{definition}

Let $\mathcal{C}, \mathcal{D}$ be presentable stable $\infty$-categories and
let $F\colon \mathcal{C} \to \mathcal{D}$ be a cocontinuous functor. 
Then if $M \in \mathcal{C} $ belongs to the localizing subcategory
generated by the $\left\{M_\alpha\right\}_{\alpha \in A} \subset \mathcal{C}$,
we can conclude that $F(M) $ belongs to the 
localizing subcategory
generated by the $\left\{F(M_\alpha)\right\}_{\alpha \in A} \subset \mathcal{D}$.
This follows easily because the preimage of a localizing subcategory is a
localizing subcategory. 
We obtain a map
\[ F_*\colon \locs(\mathcal{C})  \to \locs(\mathcal{D}) , \]
which sends a localizing subcategory $\mathcal{C}' \subset \mathcal{C}$ to the
localizing subcategory of $\mathcal{D}$ generated by $F(\mathcal{C}')$.

In this subsection, we will show that this map is an isomorphism for a faithful
Galois extension of $\e{\infty}$-rings with a \emph{connected} Galois group. 
It will be convenient to use the notion of \emph{descendability} \cite[\S
3-4]{galois}. 
\begin{definition} 
A morphism of $\e{\infty}$-rings $A \to B$ is \emph{descendable}
if the thick $\otimes$-ideal that $B$ generates in $\md(A)$ is all of
$\md(A)$.
\label{def:desc}
\end{definition} 

 Given an $A$-module $M$, one can form the cobar construction, which
is an augmented cosimplicial object
\[ M \to \left( M \otimes_A B \rightrightarrows M \otimes_A B \otimes_A B
 \triplearrows \dots \right).  \]
 If $A \to B$ is descendable, then $M$ is the totalization of the above cobar
 construction; moreover, $M$ is a retract of the partial totalization at a
 finite stage. It follows, for example, that the thick $\otimes$-ideals in
 $\md(A)$ that $M$ and $M \otimes_A B$ generate are equal.

\begin{proposition} 
\label{desc:inj}
Let $R \to R'$ be a morphism of $\e{\infty}$-rings which is descendable. 
Then the map $\locs( \md(R)) \to \locs(\md(R'))$ obtained by base-change is
injective, and has a section $\locs( \md(R')) \to \locs(\md(R))$ obtained by
restriction of scalars.
\end{proposition} 
\begin{proof} 
We have a forgetful right adjoint functor $\md(R') \to \md(R)$ given by
restriction of scalars, which is also
cocontinuous. We claim that this induces a {section} to the map 
$\locs( \md(R)) \to \locs(\md(R'))$.
Equivalently, we need to show that if $M \in \md(R)$, then the $R$-modules $M$
and $R' \otimes_R M$ generate the same localizing subcategory of $\md(R)$. 
This follows because, as remarked above, $M$ and $R' \otimes_R M$ even generate the
same \emph{thick $\otimes$-ideal} in $\md(R)$. 
\end{proof}

\begin{proposition} 
\label{sectionp}
Let $R$ be an $\e{\infty}$-ring and let $X$ be a connected finite complex. 
Then base-change $\md(R) \to \md(F(X_+, R))$ induces an isomorphism
$\locs(\md(R)) \simeq \locs( \md( F(X_+, R)))$ (whose inverse is induced by
the restriction of scalars functor).
\end{proposition} 
\begin{proof} 
Choose a basepoint $\ast \in X$, so that we obtain a map $s\colon F(X_+, R) \to R$ of
$\e{\infty}$-rings.
The morphism 
$\locs(\md(R)) \to \locs( \md( F(X_+, R)))$
admits a section since $s$ is a section of $R \to F(X_+, R)$. 
This section 
$$\phi\colon \locs(  \md( F(X_+, R))) \to \locs(\md(R))$$
comes from the functor  
\[ \md(F(X_+, R)) \xrightarrow{\otimes_{F(X_+, R)} R }  \md(R)  , \]
given by extension of scalars along $s$.
However, $s\colon F(X_+, R) \to R$ is descendable \cite[Prop. 3.34]{galois}, so $\phi$ is
also an injection  by \Cref{desc:inj}.  Therefore, $\phi$ is an inverse
to the map in question.
\end{proof}

\begin{theorem} 
Let $R$ be an $\e{\infty}$-ring and let $G$ be a
topological group with the homotopy type of a 
connected finite complex. Let $R \to R'$ be a faithful $G$-Galois
extension.
Then extension and restriction of scalars induce isomorphisms $\locs(
\md(R)) \simeq \locs( \md(R'))$ inverse to one another.
\end{theorem} 
\begin{proof} 
It suffices to show that the restriction of scalars functor
\[ \md(R') \to \md(R)  \]
induces an \emph{injection} on $\locs$, by \Cref{desc:inj}. 

Choose a descendable $\e{\infty}$-$R$-algebra $\widetilde{R}$ such that $R'
\otimes_R \widetilde{R} \simeq F(G_+, \widetilde{R})$. Then we have  a
commutative diagram
\[ \xymatrix{
\md(R') \ar[d] \ar[r]^{\otimes_R \widetilde{R}} &  \md(R' \otimes_R
\widetilde{R}) \ar[d]  \\
\md(R) \ar[r]^{\otimes_R \widetilde{R}} &  \md(\widetilde{R})
}.\]
The horizontal arrows are given by extension of scalars along $R \to \widetilde{R}$
and the vertical arrows come from restriction of scalars. 
Applying $\locs$, we obtain a commutative diagram
\[ \xymatrix{
\locs(\md(R')) \ar[d] \ar@{(->}[r]^{\otimes_R \widetilde{R}} &  \locs( \md(R' \otimes_R
\widetilde{R})) \ar[d]^{\simeq}  \\
\locs(\md(R)) \ar@{(->}[r]^{\otimes_R \widetilde{R}} &  \locs(\md(\widetilde{R}))
}\]
The left vertical arrow is an isomorphism by 
\Cref{sectionp}, and the horizontal arrows are injections by \Cref{desc:inj}. 

The diagram shows that restriction of scalars induces an injection $\locs(
\md(R')) \hookrightarrow \locs( \md(R))$. 
By \Cref{desc:inj}, it is also a surjection and is therefore an isomorphism
(with inverse induced by extension of scalars) as desired. 
\end{proof} 

The above should be compared with \cite[Th. 4.4]{BIK}, which is the
necessary comparison tool in the Benson-Iyengar-Krause proof. 

\subsection{The classification}

We note the stratification result for cochains on the classifying space
of the torus. 
Here the argument uses the existence of ``residue fields'' as in
\cite{angeltveit}. 
\begin{proposition}[{cf. \cite[Th. 5.2]{BIK}}] 
Let $k$ be a field of characteristic $p$. 
The localizing subcategories of $\md(k^{h \mathbb{T}^n})$ (resp. $\md(k^{t
\mathbb{T}^n})$) are in bijection with the
subsets of the set of homogeneous prime ideals of $\pi_*(k^{h \mathbb{T}^n})$
(resp. those not containing the irrelevant ideal).
\end{proposition} 

Let $A$ be an abelian $p$-group and construct the $\mathbb{T}^n$-action on
$\st{A} \simeq \md(k^{tA})$ as before, as well as the map $k^{t \mathbb{T}^n}
\to k^{tA}$, which is a Galois extension with Galois group isomorphic to
$\mathbb{T}^n$. 
Combining this result with the previous subsection,  
\Cref{galoistate1}, and \Cref{basicdescresult}, one obtains: 

\begin{theorem}[Benson-Iyengar-Krause \cite{BIK}] 
Let $A$ be an abelian $p$-group of rank $n$.
The functor $\md( k^{hA}) \to \md(k^{h \mathbb{T}^n})$ induced by restriction
of scalars along $k^{h \mathbb{T}^n} \to k^{hA}$ induces an isomorphism on 
$\locs$. 
As a result, the localizing subcategories of $\md(k^{hA})$ are in bijection with the
subsets of the set of homogeneous prime ideals of $\pi_*(k^{h \mathbb{T}^n})$.
\end{theorem} 

\begin{theorem}[Benson-Iyengar-Krause \cite{BIK}] 
Let $A$ be an abelian $p$-group of rank $n$.
The functor $\st{A} \to \md(k^{t \mathbb{T}^n})$ induced by restriction
of scalars along $k^{t \mathbb{T}^n} \to k^{tA}$ induces an isomorphism on 
$\locs$. 
As a result, the localizing subcategories of $\st{A}$ are in bijection with the
subsets of the set $\mathbb{P}^{n-1}_k$ of homogeneous prime ideals of $\pi_*(k^{t \mathbb{T}^n})$
which do not contain the irrelevant ideal.
\end{theorem}

\bibliographystyle{amsalpha}
\bibliography{Dadetheorem}
\end{document}